\documentclass{article}
\usepackage{fullpage,amssymb,amsmath,amsthm,amsfonts,algorithm,algorithmic,graphicx, subfigure,color,enumitem,ulem}

\newcommand{\bmat}{\left[ \begin{array}}
\newcommand{\emat}{\end{array} \right]}
\newcommand{\diag}{{\rm diag}}

\newcommand{\ignore}[1]{}
\newtheorem{theorem}{Theorem}[section]

\newtheorem{proposition}[theorem]{Proposition}
\newtheorem{lemma}[theorem]{Lemma}
\newtheorem{corollary}[theorem]{Corollary}

\theoremstyle{definition}
\newtheorem{definition}[theorem]{Definition}

\newtheorem{remark}[theorem]{Remark}

\newcommand{\smin}{s_{r,n}}
\newcommand{\fmin}{f_{r,n}}

\usepackage{upgreek}
\usepackage{epstopdf}
\DeclareGraphicsExtensions{.eps}
\title{A Generalized Randomized Rank-Revealing Factorization}

\author{Grey Ballard\thanks{Department of Computer Science, Wake Forest University, Winston Salem NC 27109},
James Demmel\thanks{Mathematics Department and CS Division,
University of California, Berkeley, CA 94720.},
Ioana Dumitriu\thanks{Mathematics Department, University of Washington, Seattle, WA 98195.}, and 
Alexander Rusciano \thanks{Mathematics Department, University of California, Berkeley, CA 94720}}

\begin{document}

\maketitle

\begin{abstract}
We introduce a Generalized Randomized QR-decomposition that may be applied to arbitrary products of matrices and their inverses, without needing to explicitly compute the products or inverses.  This factorization is a critical part of a communication-optimal spectral divide-and-conquer algorithm for the nonsymmetric eigenvalue problem.  In this paper, we establish that this randomized QR-factorization satisfies the strong rank-revealing properties.  We also formally prove its stability, making it suitable in applications.  Finally, we present numerical experiments which demonstrate that our theoretical bounds capture the empirical behavior of the factorization.
\end{abstract}

\section{Introduction} \label{Intro} 

Rank-revealing factorizations have been around for a long time, including
\cite{Golub65a}, which introduced the world to {\bf QR} with
pivoting to solve least-squares problems. Since then, many other algorithms have
been proposed, among which we mention \cite{Ipsen-Chandrasekaran},
\cite{Pan-Tang}; for a more complete list the reader can refer to
\cite{GE96}. While they all perform very well most of them
time, the only one that stably produces a strong rank-revealing factorization
(in the \cite{GE96} sense, defined in the next section) in an arithmetic complexity comparable to
{\bf QR} belongs to Ming Gu and Stanley Eisenstat \cite{GE96}. 

Recently, the idea of using randomized algorithms for rank approximation (or
more generally for low-rank approximations of matrices) has received a
lot of attention due to the applications in signal
processing and information technology, for example \cite{Lamare2009} and \cite{Markovsky}.  For a good
overview of the types of algorithms involved, see \cite{Halko-Tropp-Martinsson}. 

We provide here an analysis of the (``Randomized URV'') factorization, or {\bf RURV}, which will allow
us to prove that it has the following three properties.
\begin{enumerate}[label=\arabic*.]
\item It is strong (in the Gu-Eisenstat sense, which will be explained in
  Section \ref{defs}). In particular, it is almost as
  strong as the best existing deterministic rank-revealing
  factorization of Gu and Eisenstat \cite{GE96}; 
\item It is communication-optimal. It uses only {\bf QR} and and matrix multiplication, and thus both its arithmetic complexity and its communication complexity are asymptotically the same as {\bf QR} and matrix multiplication.
\item If the information desired is related to the invariant
  subspaces, it can be applied to a product of matrices and inverses of
  matrices \emph{without the need to explicitly calculate any
    products or inverses}. 
\end{enumerate}

To place these three properties in context, we compare with recent trends in the randomized numerical linear algebra literature.  In work focused on sketching, such as the approaches in the overview \cite{Halko-Tropp-Martinsson}, it is customary
to make the assumption that the rank is small, generally much smaller
than the size of the matrix. In such cases, the speedups achieved by
the algorithms in \cite{Halko-Tropp-Martinsson} and others like them over {\bf QR} is significant, both in
terms of arithmetic complexity and other features, like parallelization; naturally, the
results can be achieved only with (arbitrarily) high probability. The downside is that this literature largely focuses on taking advantage of matrices with low numerical rank, and quickly producing low-rank approximations of such matrices.  Such developments are insufficient for effective use as a numerical subroutine within the communication optimal generalized eigenvalue algorithm, which at minimum require parts of the first and third properties mentioned above.

Other recent approaches using randomization include \cite{DG17} and \cite{MQ+17}.  These works recognize that QR with column pivoting tends to have good rank-revealing properties, but that column pivoting induces extra communication.  They build on this realization by using randomization to guide pivot selection during the QR algorithm and introduce block pivoting strategies.  However, in contrast to our work and earlier work such as \cite{GE96}, theoretical bounds are not provided.  We therefore emphasize that our {\bf RURV} is communication optimal while maintaining the key theoretical properties of strong rank-revealing QR-factorizations.  Moreover, {\bf RURV} is conceptually simple, as it depends only on the existence of communication-avoiding {\bf QR} algorithms, which were popularized in \cite{DGHL12}.

A subset of the authors introduced {\bf RURV} in \cite{DeDuHo}, for the purpose of using it as a building block for a divide-and-conquer
eigenvalue computation algorithm whose arithmetic complexity was shown
to be the
same as that of matrix multiplication. The analysis of {\bf RURV}
performed at the time was not optimal, and the authors of \cite{DeDuHo} had not realized that
{\bf RURV} has the third property listed above. This property makes {\bf RURV} unique among
  rank-revealing factorizations, as far as we
  can tell, and it is crucial in the complexity analysis of the aforementioned
  divide-and-conquer algorithm for nonsymmetric eigenvalue
  computations \cite{BDD11}.  
  
The rest of the paper is structured as follows: in Section \ref{defs}
we give the necessary definitions and the algorithms we will use, and
Section \ref{prelims} proves some necessary probability results;
Section \ref{an_rurv} deals with the analysis of {\bf RURV}, the short
Section \ref{an_grurv} generalizes the algorithm to work for a product of
matrices and inverses, and Section \ref{numerics} presents some numerical
experiments validating the correctness and tightness of the results of Sections \ref{an_rurv} and \ref{an_grurv}.

\section{Randomized Rank-Revealing Decompositions} \label{defs}

Let $A$ be an $n \times n$ matrix with singular values $\sigma_1 \geq \sigma_2 \geq \ldots \geq \sigma_n$, and assume that there is a ``gap'' in the singular values at level $r$, that is, $\sigma_r/\sigma_{r+1} \gg 1$.

Informally speaking, a decomposition of the form $A = URV$ is called \emph{rank revealing} if the following conditions are fulfilled:
\begin{itemize}
\item[1)] $U$ and $V$ are orthogonal/unitary and $R = \left [ \begin{array}{cc} R_{11} & R_{12} \\  & R_{22} \end{array} \right ]$ is upper triangular, with $R_{11}$  $r \times r$ and  $R_{22}$  $(n-r) \times (n-r)$;
\item[2)] $\sigma_{min}(R_{11})$ is a ``good'' approximation to $\sigma_r$ (at most a factor of a low-degree polynomial in $n$ away from it),
\item[(3)] $\sigma_{max}(R_{22})$ is a ``good'' approximation to $\sigma_{r+1}$ (at most a factor of a low-degree polynomial in $n$ away from it);
\item[(4)] In addition, if $||R_{11}^{-1}R_{12}||_2$ is small (at most a low-degree polynomial in $n$), then the rank-revealing factorization is called  \emph{strong} (as per \cite{GE96}).

\end{itemize}

Rank revealing decompositions are used in rank determination \cite{stewart84}, least squares computations \cite{CH92}, 
condition estimation \cite{bischof90a}, etc.,
as well as in divide-and-conquer algorithms for eigenproblems.
For a good survey paper, we recommend \cite{GE96}.

In the paper \cite{DeDuHo}, the authors proposed a \emph{randomized}
rank revealing factorization algorithm \textbf{RURV}. Given a matrix
$A$, the routine computes a decomposition $A = URV$ with the property
that $R$ is a rank-revealing matrix; the way it does it is by
``scrambling'' the columns of $A$ via right multiplication by a
uniformly random orthogonal (or unitary) matrix $V^{H}$ (the
uniform distribution over the manifold of unitary/orthogonal
matrices is known as Haar). One way to obtain a
Haar-distributed random matrix is to start from a matrix of
independent, identically distributed normal variables of mean $0$ and
variance $1$ (denoted, here and throughout the paper, by $N(0,1)$),
and to perform the \textbf{QR} algorithm on it. The orthogonal/unitary matrix $V$ obtained through this procedure is Haar distributed.

It is worth noting that there are in the literature other ways of obtaining
Haar-distributed matrices, and some involve using fewer random bits
and/or fewer arithmetic operations; we have chosen to use this one
because it is simple and communication-optimal (as it only involves
one {\bf QR} operation, which can be performed optimally from a
communication perspective both sequentially and in parallel \cite{DGHL12,BDGJ+14,BDGJK}.  
On a practical side, we note that using less randomness would not incur any
significant overall savings, as the total arithmetic cost is much higher than the cost of
generating $n^2$ normal random variables.

Performing \textbf{QR} on the resulting matrix $A V^{H}=:\hat{A} = UR$ yields two matrices, $U$ (orthogonal or unitary) and $R$ (upper triangular), and it is immediate to check that $A = URV$. 

\begin{algorithm}[H]
\protect\caption{Function $[U, R, V] =$\textbf{RURV}$(A)$, computes a randomized rank revealing decomposition $A = URV$, with $V$ a Haar matrix.} 
\begin{algorithmic}[1]
\label{rurv}
\STATE Generate a random matrix $B$ with i.i.d. $N(0,1)$ entries. 
\STATE $[V, \hat{R}] = $\textbf{QR}$(B)$.
\STATE $\hat{A} = A \cdot V^{H}$.
\STATE $[U,R] =$ \textbf{QR}$(\hat{A})$.
\STATE Output $R$, $[U,R]$, or $[U, R, V]$.
\end{algorithmic}
\end{algorithm}

We also define the routine \textbf{RULV}, nearly identical to
\textbf{RURV}, which performs the same kind of computation (and
obtains a rank revealing decomposition of $A$), but uses \textbf{QL}
instead of \textbf{QR}, and thus obtains a lower triangular matrix in
the middle, rather than an upper triangular one. \textbf{red}{(Note one can think
of the decomposition \textbf{RULV} as being the transpose of the
decomposition \textbf{RURV} performed on $A^H$.)}

Given \textbf{RURV} and \textbf{RULV}, we now can give a method to
find a randomized rank-revealing factorization for a product of
matrices and inverses of matrices, \emph{without actually computing
  any of the inverses or matrix products}. This is a very interesting
and useful procedure in itself, and at the same time it is crucial in
the analysis of a communication-optimal Divide-and-Conquer algorithm
for the non-symmetric eigenvalue problem presented in \cite{BDD11}. 

Suppose we wish to stably find a randomized rank-revealing factorization $M_k = URV$ for the matrix $M_k = A_1^{m_1} \cdot A_2^{m_2} \cdot \ldots A_k^{m_k}$, where $A_1, \ldots, A_k$ are given matrices, and $m_1, \ldots, m_k \in \{-1,1\}$, without actually computing $M_k$ or any of the inverses. 

Essentially, the method performs \textbf{RURV} or, depending on the power, \textbf{RULV}, on the last matrix of the product, and then uses a series of \textbf{QR}/\textbf{RQ} to ``propagate'' an orthogonal/unitary matrix to the front of the product, while computing factor matrices from which (if desired) the upper triangular $R$ matrix can be obtained. A similar idea was explored by G.W. Stewart in \cite{Stewart95} to perform graded \textbf{QR}; although it was suggested that such techniques can be also applied to algorithms like $\textbf{URV}$, no randomization was used. 

The algorithm is presented in pseudocode below. For the proof of correctness, see Lemma \ref{lem:grurv}.

\begin{algorithm}
\protect\caption{Function $U =$\textbf{GRURV}$(k; A_1, \ldots, A_k; m_1, \ldots, m_k)$, computes a randomized rank-revealing decomposition $UR_1^{m_1}\cdots R_i^{m_i} V = A_1^{m_1} \cdot A_2^{m_2} \cdot \cdots A_k^{m_k}$, where $m_1, \ldots, m_k \in \{-1,1\}$.} 
\label{grurv}\begin{algorithmic}[1]
\IF{$m_k = 1$,}
\STATE $[U,R_k,V] = \textbf{RURV}(A_k)$
\ELSE
\STATE $[U,L_k,V] = \textbf{RULV}(A_k^{H})$
\STATE $R_k =L_k^{H}$ 
\ENDIF
\STATE $U_{current} = U$
\FOR{$i = k-1$ downto $1$}
\IF{$m_i = 1$,}
\STATE $[U,R_i] = \textbf{QR}(A_i \cdot U_{current})$
\STATE $U_{current} = U$
\ELSE 
\STATE $[U,R_i] = \textbf{RQ}(U_{current}^{H} \cdot A_i)$
\STATE $U_{current} = U^H$
\ENDIF
\ENDFOR
\RETURN $U_{current}$, optionally $V, R_1, \ldots, R_k$
\end{algorithmic}
\end{algorithm}


\section{Smallest singular value bounds} \label{prelims}

The estimates for our main theorem are based on the following result,
a more general case of which can be found in \cite{Dumitriu12b}; in
particular, the following is a consequence of Theorem 3.2 and Lemma 3.5. 

\begin{definition}
Let $\smin$ be a random variable denoting the smallest singular value of an $r\times r$ corner of an $n\times n$ real Haar matrix.
\end{definition}

\begin{proposition} The probability density function (pdf) of $\smin$, with $r < n/2$, is given by
\[
\fmin(x) = c_{r,n} \frac{1}{\sqrt{x}} (1 - x)^{\frac{1}{2}r(n-r)-1}  {_{2}F_{1}} \left ( \frac{1}{2}(n-r-1), \frac{1}{2}(r-1); \frac{1}{2}(n-1)+1; 1-x \right)~,
\]
	where ${_{2}F_{1}} $ is the ordinary hypergeometric function \cite{Abr_Steg}, and
\[
c_{r,n} = \frac{\frac{1}{2}r(n-r) ~\Gamma \left( \frac{1}{2}(n-r+1) \right) ~\Gamma \left ( \frac{1}{2}(r+1) \right)}{\Gamma \left (\frac{1}{2} \right) ~\Gamma \left ( \frac{1}{2} (n+1) \right)}~.
\]
\end{proposition}

This Proposition allows us to estimate very closely the probability
that $\smin$ is small. In particular, the correct scaling for the
asymptotics of $\smin$ under $r$ and$\slash$or $n \rightarrow \infty$
was proved in \cite{Dumitriu12b} to be $\sqrt{r(n-r)}$ (that is, $\smin \sqrt{r(n-r)} = O(1)$ almost surely), which means that the kind of upper bounds one should search for $\smin$ are of the form ``$\smin \leq a/\sqrt{r(n-r)}$'' for some constant $a$. This constant $a$ will depend on how confident we want to be that the inequality holds; if we wish to say that the inequality fails with probability $\delta$, then we will have $a$ as a function of $\delta$. 

\begin{lemma} 
\label{low_bd}
Let $\delta>0$, $r, (n-r)>30$; then the probability that $\smin \leq
\frac{\delta}{\sqrt{r(n-r)}}$ is $\mathbb{P}\left [ \smin \leq \frac{\delta}{\sqrt{r(n-r)}} \right] \leq  2.02 \delta$.
\end{lemma}
\begin{proof} 

What we essentially need to do here is find an upper bound on $\fmin$ which, when integrated over small intervals next to $0$, yields the bound in the Lemma.

We will first upper bound the term $(1-x)^{\frac{1}{2}r(n-r)-1}$ in the expression of $\fmin(x)$ by $1$.

Secondly, we note that the hypergeometric function has all positive arguments, and hence from its definition, it is  monotonically decreasing from $0$ to $1$, and so we bound it by its value at $x=0$. As per \cite[Formula 15.1.20]{Abr_Steg}, 
\[
_{2}F_{1} \left ( \frac{1}{2}(n-r-1), \frac{1}{2}(r-1); \frac{1}{2}(n-1)+1; 1\right) = \frac{\Gamma \left (\frac{1}{2} (n+1) \right) \Gamma\left(\frac{3}{2}\right)}{\Gamma \left ( \frac{1}{2}(n-r+2)\right) \Gamma \left ( \frac{1}{2} (r+2) \right) }~,
\]
and after some obvious cancellation we obtain that 
\[
\fmin(x) \leq \frac{1}{2} r(n-r) \cdot \frac{\Gamma \left ( \frac{1}{2}(n-r+1)\right)}{\Gamma \left (\frac{1}{2}(n-r+2)\right)} \cdot \frac{\Gamma \left(\frac{1}{2}(r+1)\right)}{\Gamma \left ( \frac{1}{2} (r+2)\right)} \cdot \frac{1}{\sqrt{x}}.
\]

The following expansion can be derived from Stirling's formula and is given as a particular case of \cite[Formula 6.1.47]{Abr_Steg} (with $a = 0,~b = 1/2$):
\[
z^{1/2} \frac{\Gamma(z)}{\Gamma(z+ 1/2)} = 1 + \frac{1}{8z} + \frac{1}{128z^2} + o\left(\frac{1}{z^2}\right)~,
\]
as $z$ real and $z \rightarrow \infty$. 
In particular, $z>30$ means that $z^{1/2} \frac{\Gamma(z)}{\Gamma(z+ 1/2)} < 1.01$. 

Provided that $r, n-r >30$, we thus have
\[
\fmin(x) \leq 1.01 \cdot \sqrt{ r(n-r)} \sqrt{\frac{r(n-r)}{(r+1)(n-r+1)}} \frac{1}{\sqrt{x}}~,
\]
and so
\[
\fmin(x) \leq 1.01  \cdot \sqrt{ r(n-r)}  \cdot \frac{1}{\sqrt{x}}~.
\]

Note that this last inequality allows us to conclude the following:
\begin{eqnarray*}
\mathbb{P}\left [ \smin \leq \frac{\delta}{\sqrt{r(n-r)}} \right] & = & \mathbb{P}\left [ \smin^2 \leq \frac{\delta^2}{r(n-r)} \right] \\
& \leq &  1.01 \int_0^{\frac{\delta^2}{r(n-r)}} \sqrt{r(n-r)} \frac{1}{\sqrt{t}} dt  \\
& = &  2.02 \delta. 
\end{eqnarray*}

\end{proof}

As an immediate corollary to Lemma \ref{low_bd} we obtain the following result, which is what we will actually use in our calculations.

\begin{corollary} 
 \label{lowbd}
 Let $\delta>0$, $r, n-r>30$. Then 
$$\mathbb{P}\left [ \frac1\smin \leq \frac{2.02}{\delta}\sqrt{r(n-r)} \right] \geq 1-\delta~.$$
\end{corollary}

\section{Analysis for RURV} \label{an_rurv}

\subsection{Bounding the probability of failure for RURV} \label{an_rurv_bounds}

It was proven in \cite{DeDuHo} that, with high probability, \textbf{RURV} computes a good rank revealing decomposition of $A$ in the case of $A$ real. Specifically, the quality of the rank-revealing decomposition depends on computing the asymptotics of $\smin$, the smallest singular value of an $r \times r$ submatrix of a Haar-distributed orthogonal $n \times n$ matrix. All the results of \cite{DeDuHo} can be extended verbatim to Haar-distributed unitary matrices; however, the analysis employed in \cite{DeDuHo} is not optimal. Using the bounds obtained for $\smin$ in the previous section, we can improve them here. 

We will tighten the argument to obtain one of the upper bounds for $\sigma_{max}(R_{22})$. In addition, the result of \cite{DeDuHo} states only that \textbf{RURV} is, with high probability, a rank-revealing factorization. Here we strengthen these
results to argue that it is actually a \emph{strong} rank-revealing
factorization (as defined in the Introduction), since with high probability $||R_{11}^{-1} R_{12}||$ will be small. 

In proving Theorem \ref{thm_rurv}, we require two lemmas that we state here and prove afterwards.
\begin{lemma}\label{lemma:R22bound}
Let $A$ be an $n \times n$ matrix whose SVD is $A = P \Sigma Q^H$, and
  with singular values $\sigma_1, \ldots, \sigma_n$. 

Let $R$ be the matrix produced by the \textbf{RURV} algorithm on $A$, \emph{in
    exact arithmetic}, so that $UR = AV^H$.  Then defining $X = Q^H V^H$,

\begin{eqnarray}
\label{doii} \sigma_{\max} (R_{22}) \leq \sigma_{\min}(X_{11})^{-1} \sigma_{r+1}~,\\
\end{eqnarray}
where $X_{11}$ is the upper-left $r \times r$ submatrix.
\end{lemma}

\begin{lemma}\label{lemma:kernel_condition}
Carrying over the notation of Lemma \ref{lemma:R22bound},
\[
\|R_{11}^{-1} R_{12}\|_2 \leq 3\sigma^{-1}_{\min}(X_{11}) + 6 \frac{\sigma_{r+1}}{\sigma_r}\sigma^{-3}_{\min}(X_{11})~.
\]
\end{lemma}

We are ready for the main result.

\begin{theorem} \label{thm_rurv}  Let $A$ be an $n \times n$ matrix
  with singular values $\sigma_1, \ldots, \sigma_r, \sigma_{r+1},
  \ldots, \sigma_n$. Let $1>\delta>0$.
Let $R$ be the matrix produced by the \textbf{RURV} algorithm on $A$, \emph{in
    exact arithmetic}.  Assume that $r, n-r > 30$. 

Then with probability $1-\delta$, the following three events occur:
\begin{eqnarray}
\label{unu} \frac{\delta}{2.02} \frac{\sigma_{r}}{\sqrt{r (n-r)}} & \leq & \sigma_{\min}(R_{11}) \leq \sigma_r ~,\\ 
\label{doi} \sigma_{r+1} & \leq & \sigma_{\max} (R_{22}) \leq 2.02 \frac{\sqrt{r(n-r)}}{\delta}\sigma_{r+1}~,\\
\label{trei} ||R_{11}^{-1} R_{12}||_2 & \leq & \frac{6.1\sqrt{r(n-r)}}{\delta} + \frac{\sigma_{r+1}}{\sigma_{r}} \frac{50 \sqrt{r^3(n-r)^3}}{\delta^3}.
\end{eqnarray}
We note the upper bound in \eqref{unu} and lower bound in \eqref{doi}
always hold.  Moreover, if we additionally assume $\delta >
\sqrt{2}\cdot 1.01 \cdot n \cdot
\frac{\sigma_{r+1}}{\sigma_r}$, then we can strengthen \eqref{trei} to 
\begin{equation}
\label{four} ||R_{11}^{-1} R_{12}||_2 \leq \frac{4.04}{\delta} \cdot \sqrt{r (n-r)} + 1
\end{equation}
\end{theorem}

\begin{remark} The factor $\sqrt{r(n-r)}$ in the equations
  \eqref{unu}, \eqref{doi}, matches the best deterministic algorithms
  up to a constant. When the gap is large enough so that
  $\frac{\sigma_{r+1}}{\sigma_{r}}$ is $O(1/n)$ with some small
  constant, so that the additional hypothesis applies, \eqref{four} also matches the best deterministic algorithms up to a constant.  Even when the gap is small, \eqref{trei} shows the factorization is strong with high probability.
 \end{remark}

\begin{proof} To prove this theorem, we will rely on Lemma
  \ref{lemma:R22bound} and Lemma \ref{lemma:kernel_condition}. For the
  sake of argument flow, we have moved the proofs of these
  lemmas to the end of the section. 
  
There are two cases of the problem, $r \leq n/2$ and $r> n/2$. 
Let $V$ be the Haar matrix used by the algorithm. 
From \cite[Theorem 2.4-1]{GVL12}, 
the singular values of $V[1:r, 1:r]$ when $r>n/2$ consist 
of $(2r-n)$ $1$'s and the singular values of $V[(r+1):n ,(r+1):n]$. 
Thus, the case $r>n/2$ reduces to the case $r \leq n/2$.


The upper bound in inequality \eqref{unu} and the lower bound in inequality \eqref{doi} follow from the Cauchy interlace theorem (see \cite[Theorem 7.3.9]{horn_johnson1}). 
The lower bound in inequality \eqref{unu} follows immediately from \cite[Theorem 5.2]{DeDuHo} and Corollary \ref{lowbd}.
We provide proofs of the upper bounds of inequalities \eqref{doi} and \eqref{trei}, below. 

Theorem 5.2 from \cite{DeDuHo} states that 
$$\sigma_{\max}(R_{22}) \leq 3\sigma_{r+1} \cdot \frac{\smin^{-4} \cdot \left ( \frac{\sigma_1}{\sigma_r} \right)^3}{1 - \frac{\sigma_{r+1}^2}{\smin^2 \sigma_r^2}}~;$$
provided that $\sigma_{r+1} < \sigma_r \smin$.
This upper bound is lax, and we tighten it here. Note that Lemma
  \ref{lemma:R22bound} establishes that
\[
\sigma_{\max}(R_{22}) \leq \sigma_{r+1} \sigma_{\min}^{-1}(X_{11}) \, ,
\]
	where $X$ is Haar distributed and $X_{11}$ is its upper-left
        $r \times r$ submatrix. 
We conclude $\sigma_{\max}(R_{22}) \leq 2.02 \frac{\sqrt{r(n-r)}}{\delta} \sigma_{r+1}$ with probability $1-\delta$, by using Corollary \ref{lowbd}.  This completes the proof of (2).

To prove \eqref{trei}, we use Lemma \ref{lemma:kernel_condition},
which establishes that
\[
\|R_{11}^{-1} R_{12}\|_2 \leq 3\sigma^{-1}_{\min}(X_{11}) + 6 \frac{\sigma_{r+1}}{\sigma_r}\sigma^{-3}_{\min}(X_{11}) \, ,
\]
where again $X$ is Haar distributed.  We conclude 
\[
\|R_{11}^{-1} R_{12}\|_2 \leq 3s_{r,n}^{-1} + 6\frac{\sigma_{r+1}}{\sigma_{r}} s_{r,n}^{-3} \, ,
\]
and apply Corollary \ref{lowbd} to get the result \eqref{trei}.

\vspace*{.3cm}
It remains to show the strengthened bound \eqref{four} on $\|R_{11}^{-1}R_{12}\|_2$ when $\delta > \sqrt{2}\cdot 1.01 \cdot n \cdot \frac{\sigma_{r+1}}{\sigma_r}$.  We use the following notation. Let $A = P \Sigma Q^{H} = P \cdot \diag(\Sigma_1, \Sigma_2) \cdot Q^{H}$ be the singular value decomposition of $A$, where $\Sigma_1 = \diag(\sigma_1, \ldots, \sigma_r)$ and $\Sigma_2 = \diag(\sigma_{r+1}, \ldots, \sigma_n)$. Let $V^H$ be the random unitary matrix in \textbf{RURV}, so that $A=URV$. Then $X = Q^{H}V^{H}$ has the same distribution as $V^{H}$, by virtue of the fact that $V$'s distribution is uniform over unitary matrices. 

Write 
\[
X = \left [ \begin{array}{cc} X_{11} & X_{12} \\ X_{21} & X_{22} \end{array} \right ]~,
\]
where $X_{11}$ is $r \times r$ and $X_{22}$ is $(n-r) \times (n-r)$. Then 
\[
U^H P \cdot \Sigma X = R~.
\]

Denote $\Sigma \cdot X = [Y_1, Y_2]$ where $Y_1$ is an $n \times r$ matrix and $Y_2$ is $n\times (n-r)$. Since $U^{H}P$ is unitary, it is not hard to check that 
\[
R_{11}^{-1} R_{12} = Y_1^{+} Y_2~,
\]
where $Y_1^{+}$ is the pseudoinverse of $Y_1$, i.e. $Y_1^{+} = (Y_1^H
Y_1)^{-1} Y_1^H$. 

There are two crucial facts that we need  to check here: one is that $R_{11}^{-1}$ actually exists, and the other is that the pseudoinverse (as defined above) is well-defined, that is, that $Y_1$ is full rank. We start with the second one of these facts.

The matrix $Y_1$ is full-rank with probability $1$. This is true due to two reasons: the first one is that the first $r$ singular values of $A$, ordered decreasingly on the diagonal of $\Sigma$, are strictly positive. The second one is that $X$ is Haar distributed, and hence Lemma \ref{low_bd} shows that $X_{11}$ is invertible with probability 1.  It follows that $Y_1^{+}$ is well-defined. 

To argue that $R_{11}^{-1}$ exists, note that $Y_1 = P^HU[R_{11};0]$ so rank($Y_1$)$=$rank($R_{11}$) as $P^HU$ is unitary.  Since $Y_1$ is full-rank, it follows that $R_{11}$ is invertible.

Having made sure that the equation relating $R_{11}^{-1} R_{12}$ and $Y_1$ is correct, we proceed to study the right hand side. From the definition of $Y$, we obtain that
\[
Y_1^H Y_1 = X_{11}^H \Sigma_{1}^2 X_{11} + X_{21}^H\Sigma_2^2 X_{21}~~, ~~~~\mbox{and}~~~~ Y_1^H Y_2 = X_{11}^H \Sigma_{1}^2 X_{12} + X_{21}^H\Sigma_2^2 X_{22}~~.
\]
Hence
\[
R_{11}^{-1} R_{12} = \left (X_{11}^{H} \Sigma_1^2 X_{11} + X_{21}^{H} \Sigma_2^2 X_{21} \right)^{-1} \left ( X_{11}^{H} \Sigma_1^2 X_{12} + X_{21}^{H} \Sigma_2^2 X_{22} \right)~.
\]
We split this into two terms. Let $T_1$ be defined as follows:
\[
T_1 := \left (X_{11}^{H} \Sigma_1^2 X_{11} + X_{21}^H \Sigma^2_2 X_{21} \right )^{-1} X_{11}^{H} \Sigma_1^2 X_{12} = X_{11}^{-1} \left ( \Sigma_1^2 + (X_{21} X_{11}^{-1})^{H} \Sigma_2^2 (X_{21} X_{11}^{-1}) \right )^{-1} \Sigma_1^2 X_{12}~,
\]
where the last equality reflects the factoring out of $X_{11}^H$ to the left and of $X_{11}$ to the right inside the first parenthesis, followed by cancellation. 
Since $X_{12}$ is a submatrix of a unitary matrix, $||X_{12}|| \leq 1$, and thus
\[
||T_1||_2 \leq ||X_{11}^{-1}||_2 \cdot  || \left (I_r + \Sigma_1^{-2} (X_{21} X_{11}^{-1})^{H} \Sigma_2^2 (X_{21} X_{11}^{-1}) \right )^{-1}||_2 \leq \frac{1}{\smin} \cdot \frac{1}{1 - \frac{\sigma_{r+1}^2}{ \smin^2 \sigma_r^2}}~,
\]
where the last inequality follows from the fact that for a matrix $A$ with $||A||<1$, $||(I-A)^{-1}|| \leq \frac{1}{1 - ||A||}$. The right hand side has been obtained by applying norm inequalities and using the fact that $||X_{11}^{-1}|| = \smin^{-1}$. The assumption on $\delta$ can be rearranged to $\frac{\delta}{\sqrt{2}\cdot 1.01\cdot n} > \frac{\sigma_{r+1}}{\sigma_r}$.  Combine this with $\frac{1}{\smin} < \frac{2.02}{\delta} \cdot \sqrt{r (n-r)}$ to get 
\begin{eqnarray} \label{sqrt_bound}
\frac{\sigma_{r+1}}{\smin \sigma_{r}} < \frac{\delta}{\sqrt{2}\cdot 1.01 \cdot n} \frac{2.02}{\delta} \cdot \sqrt{r (n-r)} = \sqrt{2} \frac{\sqrt{r(n-r)}}{n} \leq \frac{1}{\sqrt{2}}
\end{eqnarray}
We conclude that 
\begin{eqnarray} \label{r-bound}
\|T_1\|_2 \leq \frac{4.04}{\delta} \cdot \sqrt{r (n-r)}
\end{eqnarray}
We now apply similar reasoning to the second (remaining) term
\[
T_2:= \left (X_{11}^{H} \Sigma_1^2 X_{11} + X_{21}^H \Sigma^2_2 X_{21} \right )^{-1} X_{21}^H \Sigma_2^2 X_{22}~;
\]
to yield that 
\[
||T_2||_2 \leq ||X_{11}^{-1}||_2^2 \cdot ||\left ( I_r + \Sigma_1^{-2} (X_{21} X_{11}^{-1})^{H} \Sigma_2^2 (X_{21} X_{11}^{-1})\right)^{-1} ||_2 \cdot 
||\Sigma_1^{-2} ||_2 \cdot || \Sigma_2^{2}||_2 \leq \frac{\sigma_{r+1}^2}{\smin^2 \sigma_r^2} \cdot \frac{1}{1 - \frac{\sigma_{r+1}^2}{ \smin^2 \sigma_r^2}}~,
\]
because $||X_{21}||$ and $||X_{22}|| \leq 1$.  Finally, note that \eqref{sqrt_bound} together with the fact that the
function $x^2/(1-x^2)$ is increasing on $(0, \infty)$ give $\|T_2\|_2 \leq 1$. Combining this with \eqref{r-bound}, the conclusion follows. 

\end{proof}

We return now for the proofs of the two lemmas.
\begin{proof}[Proof of Lemma \ref{lemma:R22bound}]
Let $Y = \Sigma X = P^H U R$.  We begin by introducing a few block notations naturally suggested by the singular value gap, separating the first $r$ coordinates from the final $n-r$ coordinates:

\[
X = \left [ \begin{array}{cc} X_{1} & X_{2}  \end{array} \right ]~, Y = \left [ \begin{array}{cc} Y_{1} & Y_{2} \end{array} \right ] ~, \Sigma = \left [ \begin{array}{cc} \Sigma_{1} & 0 \\ 0 & \Sigma_2 \end{array} \right ]~,
\]

Note that $R$ is the upper-triangular factor resulting
  from QR-factorization of $Y$.  From this and understanding that the
QR-factorization records the Gram-Schmidt orthogonalization process,
we see that $R_{22}$ has the same singular values as
${\text{Proj}}_{{Y_{1}}_\perp} Y_{2}$, where ${Y_{1}}_{\perp}$ is any
matrix whose columns are a basis for the orthogonal complement of
$Y_1$. In particular,
$\sigma_{\max}(R_{22}) = \|{\text{Proj}}_{{Y_{1}}_\perp} Y_{2}\|_2$.  

It is also clear from $Y = \Sigma X$ that if $\Sigma$ contains $0$
elements on the diagonal, then the corresponding rows of
${\text{Proj}}_{{Y_{1}}_\perp} Y_{2}$ are $0$.  Therefore we make the assumption that there are no $0$ singular values.

We next relate $Y_{1}$ to $Y_{2}$ in order to analyze this matrix.
Using a common matrix identity in first equality and the orthogonality
of $X$ in the third equality, we see that 
\[
\text{im}({Y_{1}})_\perp = \ker(Y_{1}^H)= \ker(X_{1}^H \Sigma)= \text{im}(\Sigma^{-1} X_{2})
\]
Thus, we seek to bound $\|{\text{Proj}}_{ \Sigma^{-1} X_{2}} \Sigma
X_2\|_2$.  

We will present a coordinate-based bound. First, it is true for any invertible $U$ that ${\text{Proj}}_Y X = {\text{Proj}}_{YU} X$.  We set $U$ to be the $(n-r) \times (n-r)$ orthogonal matrix $L$ of right-singular vectors of
${\text{Proj}}_{ \Sigma^{-1} X_{2}} \Sigma X_2$. Thus,
\[
{\text{Proj}}_{ \Sigma^{-1} X_{2}} \Sigma X_2 = {\text{Proj}}_{
  \Sigma^{-1} X_{2}L} \Sigma X_2 \, ,
\]
and by the definition of $L$,
\[
\|{\text{Proj}}_{ \Sigma^{-1} X_{2}} \Sigma X_2\|_2 = \|{\text{Proj}}_{ \Sigma^{-1} X_{2}L} \Sigma X_2L[1:n-r,1]\|_2~.
\]

 To keep notation as simple as possible, we denote 
 $Z = X_2 L$ and partition $Z = (z, Z')$.  Selecting any column $z'$
 from $Z'$, observe that $(\Sigma z)^H (\Sigma^{-1} z') = z^H z' = 0$.
 Thus we can compute the needed projection by performing Gram-Schmidt
 on $\Sigma^{-1} z$ with respect to $\Sigma^{-1}Z'$, which we observe to produce $\Sigma^{-1}z-{\text{Proj}}_{ \Sigma^{-1} Z'}\Sigma^{-1} z$.
 \[
\|{\text{Proj}}_{ \Sigma^{-1} Z} \Sigma z \|_2 = \frac{(\Sigma z)^H\left(\Sigma^{-1}z-{\text{Proj}}_{ \Sigma^{-1} Z'}\Sigma^{-1} z\right)}{\|\Sigma^{-1}z-{\text{Proj}}_{ \Sigma^{-1} Z'}\Sigma^{-1} z\|_2} = \|\Sigma^{-1}z-{\text{Proj}}_{ \Sigma^{-1} Z'}\Sigma^{-1} z\|_2^{-1}
\]

We are looking for an upper bound for the latter, so it suffices to
bound the quantity $\|\Sigma^{-1}z-{\text{Proj}}_{ \Sigma^{-1}
  Z'}\Sigma^{-1} z\|_2$ away from $0$. Note that if we can restrict
the problem to the last $n-r$ coordinates and get a non-trivial lower
bound, we have achieved our goal.

For simplicity, we denote the last $(n-r)$ coordinates of $Z$ by $B$, of
$Z'$ by $B'$, and of $z$ by $b$.  The quantity
$\|\Sigma_2^{-1}b-{\text{Proj}}_{ \Sigma_2^{-1} B'}\Sigma_2^{-1}
b\|_2$ is a least squares error, specifically,
\[
\|\Sigma_2^{-1}b-{\text{Proj}}_{ \Sigma_2^{-1} B'}\Sigma_2^{-1}
b\|_2^{-1} = \left(\min\limits_{x \in
    \mathbb{R}^{n-r-1}}\|\Sigma_2^{-1} B' x - \Sigma_2^{-1} b \|_2
\right)^{-1} ~,
\]
and on the other hand, trivially, 
\[
\left(\min\limits_{x \in
    \mathbb{R}^{n-r-1}}\|\Sigma_2^{-1} B' x - \Sigma_2^{-1} b \|_2
\right)^{-1} \leq \sigma_{r+1} \left(\min\limits_{x \in \mathbb{R}^{n-r-1}}\|B' x - b \|_2 \right) ^{-1}~. 
\]

To complete the proof, we bound the quantity $\min\limits_{x \in \mathbb{R}^{n-r-1}}\|B' x - b \|_2 $ away from $0$ in terms of the smallest singular value of the lower right $(n-r) \times (n-r)$ submatrix of the original random matrix $X$.  Indeed,
\[
\min\limits_{x \in \mathbb{R}^{n-r-1}}\|B' x - b \|_2 = \min\limits_{x \in \mathbb{R}^{n-r-1}}\|B \left( \begin{array}{c} 1 \\ x \\ \end{array} \right) \|_2 \geq \sigma_{\min}(B)
\]

However, also recall that $B$ is exactly the lower $(n-r) \times
(n-r)$ block of $X_2 L$.  As $L$ is unitary and applied on the right,
we see $\sigma_{\min}(B) = \sigma_{\min}(X_{22})$.  Finally, it is not
difficult to check that the orthogonality of $X$ implies that $\sigma_{\min}(X_{22}) = \sigma_{\min}(X_{11})$.  Therefore, in total, we have shown

\[
\|R_{22}\|_2 \leq \sigma_{r+1} \sigma_{\min} (X_{11})^{-1} \, .
\]

\end{proof}

\begin{proof}[Proof of Lemma \ref{lemma:kernel_condition}]
Let $R'$ be the upper-triangular result of $\textbf{QR}(A
\tilde{V}^H)$, where $\tilde{V}^H$ is formed by swapping column $i\leq
r$ of $V^H$ with column $j+r$. This is equivalent to saying that $R'$
is the upper-triangular result of $\textbf{QR}(\Sigma \tilde{X})$, where $\tilde{X}$ swaps column $i$ with column $j+r$.  This point of view will be more helpful in the proof.  From Lemma 3.1 of \cite{GE96}, 
\begin{equation} \label{eq:gu}
\left | (R_{11}^{-1} R_{12})[i,j]\right| \leq \frac{\left|\text{det}(R_{11}')\right|}{\left|\text{det}(R_{11})\right|}~.
\end{equation}  
We are particularly interested in the case when $i = r$ and $j = n-r$,
as will become apparent below.

Since our bound is going to use the coordinate-based inequality
\eqref{eq:gu}, much as in Lemma \ref{lemma:R22bound}, it is again
useful to change coordinates to an optimal choice.  One can check that
$R_{11}^{-1} R_{12} = (\Sigma X_1)^+ \Sigma X_2$, since
(\ref{eq:gu}) can be viewed as a generalization of Cramer's Rule.
Therefore for orthogonal matrices of appropriate sizes $\bar{U}$, $\bar{V}$,
\[
\|R_{11}^{-1} R_{12}\|_2 = \|\bar{U}^H R_{11}^{-1} R_{12} \bar{V} \|_2 = \|(\Sigma X_1 \bar{U})^+(\Sigma X_2 \bar{V})\|_2~.
\]
Choosing now $\bar{U}$ and $\bar{V}$ to be given by
  the SVD of $R_{11}^{-1} R_{12}$ in appropriate column order, we can
  ensure that the norm of $R_{11}^{-1} R_{12}$ is the lower right
  entry of $\bar{U} R_{11}^{-1} R_{12}\bar{V}^H =
  (R_{11}\bar{U}^H)^{-1} R_{12}\bar{V}^H$. 

Suppose now that we bounded the
entries of $R_{11}^{-1} R_{12}$ in terms of a
function of the matrix $X$ that is invariant under
right multiplication of $X$ by a block-orthogonal matrix $\left
  ( \begin{array}{cc} \bar{U}^H & 0 \\ 0 & \bar{V} \end{array}
\right)$.  Then the bound on the bottom right entry $(r,n-r)$ of
$R_{11}^{-1} R_{12}$ would apply to the operator norm.
Hence, using Equation \eqref{eq:gu}, our the task is to bound 
\[
R_{11}^{-1} R_{12}[r,n-r] \leq \frac{\left|\text{det}(R_{11}')\right|}{\left|\text{det}(R_{11})\right|} = \frac{\left|R'[r,r]\right|}{\left|R[r,r]\right|}.
\]
where $R'$ has resulted from swapping column $r$ in $X$ with column
$n$ in $X$.

\vspace*{.3cm}
With the preliminaries over, we introduce notation to assist in
  the proof. Using Matlab notation, let $X_{11} = X[1:r,1:r]$, $X_1' = X[1:n, 1:r-1]$, $X_{11}' =
X[1:r, 1:r-1]$, $X_{21}' = X[r+1:n,1:r-1]$;  $x_r$ will denote the $r$-th
column of $X$, $x_n$ will denote the last column of $X_2$ (also of
$X$). We make use of one projection extensively, and therefore denote it by letter $\Pi := \text{Proj}_{(\Sigma
  X_1')_{\perp}} \Sigma (\cdot)$. Finally, projection onto the first $r$ coordinates is used, and we denote this by $\zeta$.

\vspace*{.3cm}
We are ready for the main part of the
lemma's proof.  First upper bound $|R'[r,r]| = \|\Pi
  \Sigma x_n\|_2$.  Let $\left(\begin{array}{c} w \\
    0 \end{array}\right)$ be the maximal right singular unit-vector of
the operator $\Pi \Sigma \zeta$.  We have

\[
|R'[r,r]| = \| \Pi \Sigma x_n \|_2 \leq \|\Pi \Sigma \left ( \begin{array}{c} w \\ 0 \end{array} \right) \|_2 + \max\limits_{\|t\|=1}\|\Pi \Sigma \left(\begin{array}{c} 0 \\ t \end{array} \right) \|_2 \leq \|\Pi \Sigma \left(\begin{array}{c} w \\ 0 \end{array} \right) \|_2 + \sigma_{r+1}
\]
To analyze this further, decompose $w=c_1v+c_2 u$ where $v \in
{X'_{11}}_{\perp}$ and $u \in \text{im}(X_{11}')$ are unit-vectors
orthogonal to each other.  We will have to control this $u$ term later
in the proof.  To this end, note that we have the following least
squares interpretation:
\begin{equation*}
\|\Pi \Sigma \left( \begin{array}{c} u \\ 0 \end{array}\right) \|_2 =
\min\limits_{y\in\mathbb{R}^{r-1}}\| \Sigma X_1' y-\Sigma
\left(\begin{array}{c} u \\ 0 \end{array} \right) \|_2 ~,
\end{equation*}
and by making $y = X_{11}'^{+}u$ and taking advantage of the fact that
$u$ is in the image of $X_{11}'$ so that $X_{11}'X_{11}'^+ u = u$, we
obtain
\begin{equation}\label{uterm}
\|\Pi \Sigma \left( \begin{array}{c} u \\ 0 \end{array}\right) \|_2 \leq \|\Sigma_2 X_{21}' X_{11}'^+ u\|_2 \leq \sigma_{r+1} \sigma_{\min}^{-1}(X_{11})~.
\end{equation}
In the second inequality above we have used that
$\sigma_{\min}(X_{11}') \geq \sigma_{\min}(X_{11})$.

Our main upper bound on $|R'[r,r]|$ is then
 \begin{equation}\label{mainupper}
|R'[r,r]| \leq  \|\Pi \Sigma \left(\begin{array}{c} v \\ 0 \end{array} \right) \|_2 + \sigma_{r+1} + \sigma_{r+1} \sigma_{\min}^{-1}(X_{11})
 \end{equation}
\vspace*{.3cm}
Due to the presence of $ \|\Pi \Sigma
  \left(\begin{array}{c} v \\ 0 \end{array} \right) \|_2$ in the bound
  above, we will need to consider two different cases depending on
  whether this term is small or large; as such, we will need to
  develop two different lower bounds on $|R[r,r]|$. We now build the the first lower bound on $|R[r,r]| = \|\Pi \Sigma x_r\|_2$.
\begin{align*}
|R[r,r]| = \|\Pi \Sigma x_r\|_2 = \min\limits_{y \in \mathbb{R}^{r-1}} \|\Sigma X_1' y - \Sigma x_r\|_2 &\geq\min\limits_{y_1 \in \mathbb{R}^{r-1}} \|\Sigma_1 X_{11}' y_1 - \Sigma_1 (x_r)_1\|_2 \\  
 &= \min\limits_{y_1 \in \mathbb{R}^{r-1}} \|\Sigma_1 X_{11} \left ( \begin{array}{c} y_1 \\ 1 \end{array} \right ) \|_2 \geq \sigma_{r}\sigma_{\min}(X_{11}) \, ,
\end{align*}
 where $(x_r)_1 = x_r[:r]$. This proves the first lower bound we need, 
\begin{equation}
|R[r,r]| \geq \sigma_{r}\sigma_{\min}(X_{11}) \label{boundoption2}
\end{equation}
 
\vspace*{.3cm} 
The second lower bound on $|R[r,r]|$ is similar in spirit, and is introduced to take care of the case when $\|\Pi \Sigma \left( \begin{array}{c}v \\ 0 \end{array} \right)\|_2$ is large.  Recall $v$ is the unique direction orthogonal to the the columns of $X_{11}'$, and additionally let $c$ be the magnitude of the projection of $x_r$ onto $\left ( \begin{array}{c} v \\ 0 \end{array} \right)$.  Again interpreting projection through least squares, 
\[
c = \min\limits_{y \in \mathbb{R}^{r-1}} \|X_{11} \left ( \begin{array}{c} y \\ 1 \end{array} \right ) \|_2 \geq \sigma_{\min}(X_{11})
\]
Now use the reverse triangle inequality,
\begin{equation}
|R[r,r]| = \|\Pi \Sigma x_r\|_2 \geq  \|\Pi \zeta\Sigma x_r\|_2 - \|\Pi (I-\zeta)\Sigma x_r\|_2 \label{triangle} \, .
\end{equation}
We need to lower bound the term  $\|\Pi \zeta\Sigma x_r\|_2
  $ and to upper bound the term $ \|\Pi (I-\zeta)\Sigma x_r\|_2$.
Decompose $\zeta x_r = \pm cv + u'$ with $u'$ in the image of $X_{11}'$.
Use the same algebra as in \eqref{uterm} to control the $u'$ term,
establishing that
\[
\|\Pi \zeta\Sigma x_r\|_2 \geq c \|\Pi \Sigma \left(\begin{array}{c}v \\ 0 \end{array}\right) \|_2 - \|\Pi \Sigma \left(\begin{array}{c} u' \\ 0 \end{array}\right) \|_2 \geq c \|\Pi \Sigma \left(\begin{array}{c}v \\ 0 \end{array}\right) \|_2  - \sigma_{r+1} \sigma_{\min}^{-1}(X_{11})~.
\]
To upper bound the second term of \eqref{triangle}, note that
\[
\|(I-\zeta)\Sigma x_r\|_2 \leq \|(I-\zeta)\Sigma\|_2 = \sigma_{r+1}~.
\]
Combining all of these observations and using the lower bound on $c$,
the second lower bound on $|R[r,r]|$ is

\begin{equation}
|R[r,r]| \geq \sigma_{\min}(X_{11})\| \Pi \Sigma\left(\begin{array}{c}v \\ 0 \end{array}\right) \|_2 -2 \sigma_{r+1} \sigma_{\min}^{-1}(X_{11})\label{boundoption1}
\end{equation}

\vspace*{.3cm}
To conclude the proof, we present two cases of the ratio between $|R'[r,r]|$, bounded in \eqref{mainupper}, and $|R[r,r]|$, bounded in \eqref{boundoption2} and \eqref{boundoption1}.  For the first case, we make the assumption $ \|\Pi \Sigma \left(\begin{array}{c} v \\ 0 \end{array} \right)\|_2 \geq 4 \sigma_{r+1} \sigma^{-2}_{\min}(X_{11})$.  In this situation, \eqref{boundoption1} is superior.  We get
\begin{align*}
\frac{|R'[r,r]|}{|R[r,r]|} &\leq \frac{\|\Pi \Sigma \left(\begin{array}{c} v \\ 0 \end{array} \right) \|_2 + \sigma_{r+1} + \sigma_{r+1} \sigma_{\min}^{-1}(X_{11})}{\sigma_{\min}(X_{11})\| \Pi \Sigma\left(\begin{array}{c}v \\ 0 \end{array}\right) \|_2 -2 \sigma_{r+1} \sigma_{\min}^{-1}(X_{11})} \\ 
& \leq \frac{1.5\|\Pi \Sigma \left(\begin{array}{c} v \\ 0 \end{array} \right) \|_2 \sigma_{\min}^{-1}(X_{11})}{\| \Pi \Sigma\left(\begin{array}{c}v \\ 0 \end{array}\right) \|_2 -2 \sigma_{r+1} \sigma_{\min}^{-2}(X_{11})}  \\
&\leq \frac{1.5\|\Pi \Sigma \left(\begin{array}{c} v \\ 0 \end{array} \right) \|_2 \sigma_{\min}^{-1}(X_{11})}{.5 \| \Pi \Sigma\left(\begin{array}{c}v \\ 0 \end{array}\right) \|_2}  = 3 \sigma_{\min}^{-1}(X_{11})
\end{align*}
The second case is $\|\Pi \Sigma \left(\begin{array}{c} v \\ 0 \end{array} \right)\|_2 < 4 \sigma_{r+1} \sigma^{-2}_{\min}(X_{11})$.  In this situation we must use \eqref{boundoption2},  
\begin{align}
\frac{|R'[r,r]|}{|R[r,r]|} &\leq \frac{\|\Pi \Sigma \left(\begin{array}{c} v \\ 0 \end{array} \right) \|_2 + \sigma_{r+1} + \sigma_{r+1} \sigma_{\min}^{-1}(X_{11})}{\sigma_{r}\sigma_{\min}(X_{11})} \\
&\leq 6\frac{\sigma_{r+1}}{\sigma_r} \sigma_{\min}^{-3}(X_{11})
\end{align}

The bound we have given depends on $\sigma_{\min}(X_{11})$, therefore satisfying the unitary invariance property we required.  To make the statement simple we have added the two bounds together in the lemma statement.

\end{proof}

\subsection{Stability of {\bf RURV}} \label{an_grurv_stability}

The following appeared in \cite{DeDuHo} as Lemma 5.4.  We include a short proof for completeness, and also because it is used to prove Theorem \ref{thm_grurv}.

\begin{theorem} \label{stable_rurv} {\bf RURV} is backward stable. \end{theorem}

\begin{proof} We need two facts: that {\bf QR} is backward stable
  (e.g., implemented via Householder reflectors) and that, while
  multiplication by a square matrix is not, in general, backward
  stable, multiplication by a square unitary matrix is (this is a
  simple exercise appearing in many Numerical Linear Algebra books, which we leave for the reader). 

We input a matrix $A$ and output two matrices, $U$ and $R$, such that (in the absence of floating point error) we should have $UR = A V^{H}$, for some $V$ unitary, with $U$ unitary and $R$ upper triangular. For backward stability, we would like to show that in practice the outputs $U$ and $R$ satisfy $(U+dU)R = (A+dA)(V+dV)^{H}$, with $V+dV$ unitary, $U + dU$ unitary, and $||dA||/||A|| = O(\epsilon_{mach})$.  We know $||dV||$ and $||dU||$ are $O(\epsilon)$ because \textbf{QR} is stable.

We start with the fact that matrix multiplication by $(V+dV)^{H}$ is backward stable; recalling the definition $A V^{H}=:\hat{A}$ with round-off, then

\[
\hat{A} = (A+(dA)_1) (V+dV)^{H}~;
\]
with $||(dA)_1||/||A|| = O(\epsilon_{mach})$; and since {\bf QR} is also stable, the output $[U,R]$ will have the property that 
\[
(U+dU)R = \hat{A} + (dA)_2~,
\]
with $||(dA)_2||/||\hat{A}|| = O(\epsilon_{mach})$. Combining these,

\[
(U+dU)R = A (V+dV)^{H} + (dA)_1 (V+dV)^{H} + (dA)_2 = (A + (dA)_1 + (dA)_2 (V+dV)) (V+dV)^{H}~.
\]
As $(V+dV)^{H}$ is orthogonal, this means that $||(dA)_2(V+dV)||/||A|| = O(\epsilon_{mach})$, and we conclude that 
\[
(U+dU)R = (A+dA) (V+dV)^{H}~,
\]
where $dA = (dA)_1+ (dA)_2(V+dV)$ has the property that $||dA||/||A||=O(\epsilon_{mach})$.
\end{proof}

\section{Analysis of {\bf GRURV}} \label{an_grurv}

In this short section we prove that, given a matrix $M_k = A_1^{m_1}
\cdot A_2^{m_2} \cdot \ldots A_k^{m_k}$, where $m_1, \ldots, m_k \in
\{-1,1\}$, and such that only the matrices $A_i$ may be available, {\bf GRURV} can be applied to get
the same rank-revealing factorization we would obtain in the case of
applying {\bf RURV} to the explicitly formed product $M_k$.  

\begin{lemma}\label{lem:grurv}
\textbf{GRURV} (\emph{Generalized Randomized URV}) computes the
rank-revealing decomposition $M_k = U_{current} R_1^{m_1} \ldots R_k^{m_k} V$. 
\end{lemma}

\begin{proof} Let us examine the case when $k = 2$ ($k>2$ results immediately through simple induction). 

Let us examine the cases:
\begin{enumerate}
\item $m_2 = 1$. In this case, $M_2 = A_1^{m_1} A_2$; the first \textbf{RURV} yields
$M_2 = A_1^{m_1} UR_2V$. \begin{enumerate}
\item if $m_1 = 1$, $M_2 = A_1 UR_2V$; performing \textbf{QR} on $A_1 U$ yields
$M_2= U_{current} R_1 R_2 V$.
\item if $m_1 = -1$, $M_2 = A_1^{-1} UR_2V$; performing \textbf{RQ} on $U^{H}A_1$ yields $M_2 = U_{current} R_1^{-1} R_2 V$.
\end{enumerate}
\item $m_2 = -1$. In this case, $M_2 = A_1^{m_1} A_2^{-1}$; the first \textbf{RULV} yields $M_2 = A_1^{m_1} U L_2^{-H} V = A_1^{m_1} U R_2^{-1} V$. \begin{enumerate}
\item if $m_1 = 1$, $M_2 = A_1 U L_2^{-H} V = A_1 U R_2^{-1} V$; performing \textbf{QR} on $A_1 U$ yields $M_2 = U_{current} R_1 R_2^{-1} V$.
\item finally, if $m_2 = -1$,  $M_2 = A_1^{-1} U L_2^{-H} V = A_1^{-1} U R_2^{-1}V$; performing \textbf{RQ} on $U^{H}A_1$ yields $M_2 = U_{current} R_1^{-1} R_2^{-1} V$.
\end{enumerate}
\end{enumerate}

Note now that in all cases $M_k = U_{current} R_1^{m_1} \ldots R_k^{m_k} V$. Since the inverse of an upper triangular matrix is upper triangular, and since the product of two upper triangular matrices is upper triangular, it follows that $R:=R_1^{m_1} \ldots R_k^{m_k}$ is upper triangular. Thus, we have obtained a rank-revealing decomposition of $M_k$; the same rank-revealing decomposition as if we have performed $QR$ on $M_k V^{H}$.
\end{proof}


This allows us to conclude the following important stability result for \textbf{GRURV}.  We note that it was claimed without proof as Theorem 2.3 in the technical report \cite{BDD11}.

\begin{theorem} \label{thm_grurv} In the absence of floating point
  error, the result of the algorithm \textbf{GRURV} is essentially the same as the result of 
\textbf{RURV} on the (explicitly formed) matrix $M_k$.  This means
that all results of Theorem \ref{thm_rurv} apply for the product
matrix $R =  R_1^{m_1} \ldots R_k^{m_k}$. \end{theorem}

Note that we may also return the matrices $R_1, \ldots, R_k$, from which the factor $R$ can later 
be reassembled, if desired.

\begin{theorem} {\bf GRURV} is backward stable. \end{theorem}

\begin{proof} Simple induction once again shows that it is sufficient
  to consider the case $k=2$.  

For the case $k=2$, we will do the calculations for $m_1=1$ and $m_2 =
\pm 1$; the other two cases can be dealt with in the same fashion. 

Note {\bf QR}, {\bf RQ}, and {\bf QL} can be performed in a backward-stable
manner, all multiplications performed are multiplications by unitary
matrices (and thus backward stable), and we have shown that {\bf RURV} (respectively {\bf RULV}, as
the only difference is the application of {\bf QL} instead of {\bf
  QR}) is also backward stable. 

Let now $m_1 = 1$ and let $[U_2, R_2]$ be the outputs of the first {\bf
	QR} operation performed by {\bf GRURV}, and $[U_1, R_1]$ be the outputs of the second. Then stability of {\bf QR} gives $(U_1+d U_1)R_1 = A_1 \cdot U_2 + d(A_1 \cdot U_2)$ for some $d( A_1 \cdot U_2)$ satisfying $||d(A_1 \cdot U_2) ||/|| A_1 \cdot U_2|| =
O(\epsilon_{mach})$.  However, the stability of {\bf RURV} ensures the existence of $dU_2'$ satisfying $||dU_2'||/||U_2|| =
O(\epsilon_{mach})$ such that $U_2 + dU_2'$ is orthogonal.  As this establishes $U_2$ has condition number $1-O(\epsilon)$, we therefore conclude $d(A_1 \cdot U_2) = (d A_1) \cdot U_2$
  for some $||dA_1||/||A_1|| =
O(\epsilon_{mach})$.  More specifically, this is for $dA_1 = d(A_1 U_2) U_2^{-1}$. Combining these steps, we have shown 
\[
(U_1+d U_1)R_1 = (A_1 + dA_1) \cdot U_2 \, .
\]
Now we break into cases:


\begin{itemize}
	\item $m_2=1$. Then a simple consequence of Theorem \ref{stable_rurv} is $U_{2}R_2 (V+dV) = A_2 + dA_2$, with $||dA_2||/||A_2||
  = O ( \epsilon_{mach})$.

Putting it all together, 
\[
	(U_1+dU_1) R_1 R_2 (V+dV) = (A_1 + dA_1) \cdot U_2 R_2 (V+dV) = (A_1+dA_1)(A_2+dA_2)~,
\]
with $||dA_1||/||A_1||
  = O ( \epsilon_{mach})$ and $||dA_2||/||A_2|| =
O(\epsilon_{mach})$, so in this case {\bf GRURV} is backward stable. 

\item $m_2 = -1$. Then $U_2 R_2^{H} (V + dV)= (A_2+ dA_2)^H$, with $||dA_2||/||A_2||
  = O ( \epsilon_{mach})$, again because {\bf RULV} is
  backward stable. By Hermitian transposing and inverting, we obtain
		that $U_2 R_2^{-1} (V+dV) = (A_2+dA_2)^{-1}$. 

Putting it all together,
\[
	(U_1+dU_1) R_1 R_2^{-1} (V+dV) = (A_1+dA_1)\cdot U_2 R_2^{-1} (V+dV) = (A_1+dA_1)(A_2+dA_2)^{-1}~,
\]
with $||dA_1||/||A_1||
  = O ( \epsilon_{mach})$ and $||dA_2||/||A_2|| =
O(\epsilon_{mach})$, so again {\bf GRURV} is backward stable.
\end{itemize}

The other two cases are dealt with in the same fashion, and simple
induction on $k$ shows that {\bf GRURV} is backward stable for any $k
\geq 2$. 
\end{proof}

\section{Numerical Experiments} \label{numerics}

In this section, we present numerical experiments to test the four bounds of Theorem \ref{thm_rurv}.
Since Theorem \ref{thm_rurv} utilizes the asymptotic behavior of singular values of submatrices of Haar matrices, becoming more accurate as the dimensions of the matrix and submatrix increase, we will perform tests over a range of matrix dimensions.

To test the effects of dimension and gap on the effectiveness of \textbf{RURV}, we set up problems with specified singular value distributions and Haar distributed left and right singular vectors.
For our experiments, we consider two types of singular value distributions:
\begin{itemize}
	\item stair step distribution: $\sigma_1=\sigma_r$, $\sigma_{r+1}=\sigma_n=1$, with a specified gap $\sigma_r/\sigma_{r+1}$;
	\item logarithmically spaced distribution: $\sigma_1= 10^{13}$, $\sigma_n =1$, $\sigma_i/\sigma_{i+1} = \sigma_j/\sigma_{j+1}$ for all $i,j \neq r$, with a specified gap $\sigma_r/\sigma_{r+1}$.
\end{itemize}
We also perform two types of experiments
\begin{itemize}
	\item fix the matrix dimension $n=1500$, and vary the gap $\sigma_r/\sigma_{r+1}$ ranging from $10^1$ to $10^{10}$;
	\item fix the gap $\sigma_r/\sigma_{r+1}=10^7$, and vary the matrix dimension $n$ ranging from 250 to 2000.
\end{itemize}
Across all experiments, we fix $r=n/2$ and repeat each experiment 1000 times, constructing matrices with the fixed singular value distribution and Haar-distributed random singular vectors.

We compare the results of these four experiments against the theoretical bounds of Theorem \ref{thm_rurv} in Figures \ref{fig:spread1varygap}, \ref{fig:spread1varydim}, \ref{fig:spread100varygap}, and \ref{fig:spread100varydim}.
In each figure, the top left plot shows the ratio $\sigma_1/ \sigma_{\text{min}}(R_{11})$ and Inequality \eqref{unu}, the top right plot shows the ratio $\sigma_{\text{max}}(R_{22}) / \sigma_{r+1}$ and Inequality \eqref{doi}, the bottom left plot shows $\|R_{11}^{-1}R_{22}\|_2$ and Inequalities \eqref{trei} and \eqref{four}, and the bottom right plot shows an example singular value distribution to illustrate the distribution as a stair step or logarithmically spaced.
The experimental results are presented as box plots, with boxes corresponding to the inter-quartile range (middle 50\%) and horizontal lines corresponding to min, median, and max.
In order to compare against the probabilistic bounds, across all experiments, we specify a fixed failure probability of $3\%$ ($\delta=0.03$) to obtain the theoretical bounds, and we plot the 97th percentile of the experimental data as a black line.
The theoretical bounds \eqref{unu}, \eqref{doi}, and \eqref{four} appear as blue lines with asterisk markers; they have values 
\begin{eqnarray*}
\frac{\sigma_r}{\sigma_{\min}(R_{11})} & \leq & \frac{2.02}{\delta} \sqrt{r (n-r)} ~,\\ 
\frac{\sigma_{\max} (R_{22})}{\sigma_{r+1}} & \leq & \frac{2.02}{\delta} \sqrt{r(n-r)}~,\\
||R_{11}^{-1} R_{12}||_2 & \leq & \frac{4.04}{\delta} \cdot \sqrt{r (n-r)} + 1 ~.
\end{eqnarray*}
For \eqref{four} from Theorem \ref{thm_rurv} to be valid, $\sigma_r/\sigma_{r+1} > \sqrt{2} \cdot 1.01 \cdot \frac{n}{\delta}$; we do not plot the bound when it does not apply (see Figures \ref{fig:spread1varygap} and \ref{fig:spread100varygap}).
Theoretical bound \eqref{trei} appears as a red line with circle markers.

Overall, we observe that the probabilistic bounds \eqref{unu}, \eqref{doi}, \eqref{four} hold up empirically and are very tight given that the black line never exceeds the blue line but remains close in all experiments.

In the experiment varying the gap for the stair step singular value distribution (Figure \ref{fig:spread1varygap}), a deterministic bound governs the behavior of the algorithm for small gaps.
The blue dotted line corresponds to the deterministic upper bounds 
\begin{eqnarray}
\label{eq:det1} \frac{\sigma_r}{\sigma_{\min}(R_{11})} \leq \frac{\sigma_r}{\sigma_n}~, \\
\label{eq:det2} \frac{\sigma_{\max} (R_{22})}{\sigma_{r+1}} \leq \frac{\sigma_1}{\sigma_{r+1}}~, \\
\label{eq:det3} \|R_{11}^{-1} R_{12}\|_2 \leq \frac{\sigma_1}{\sigma_n}~.
\end{eqnarray}
We also plot these bounds in the same experiment for the logarithmically spaced distribution (Figure \ref{fig:spread100varygap}), but in that case they are very loose.

To get a more detailed view of the distributions of these quantities than the box plots, we provide histograms for the empirical data for two examples from the experiments, one for each singular value distribution.
In the histogram plots, we also include vertical bars to show the 97th percentile (black line) and the theoretical bound (blue line with asterisk marker).
Figure \ref{fig:spread1varydim-hist} shows the logarithms of the quantities $\sigma_1/ \sigma_{\text{min}}(R_{11})$, $\sigma_{\text{max}}(R_{22}) / \sigma_{r+1}$, and $\|R_{11}^{-1}R_{22}\|_2$ for dimension $n=1500$ and gap $\sigma_r/\sigma_{r+1}=10^7$.
Figure \ref{fig:spread100varydim-hist} shows the same quantities for the logarithmically space distribution with the same dimension ($n=1500$) and gap ($\sigma_r/\sigma_{r+1}=10^7$).

\newcommand{\figscale}{0.75}

\begin{figure}
\centering
\includegraphics[scale=\figscale]{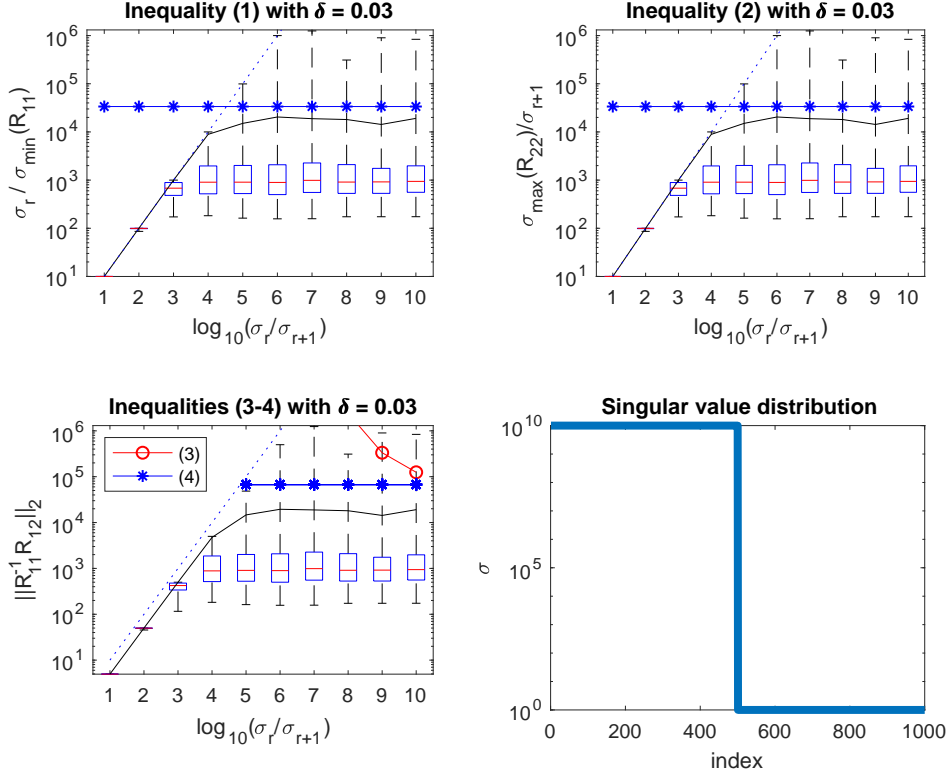}
\caption{Experimental results for test matrices with stair step distribution with dimension $n=1500$ and varying gap $\sigma_r/\sigma_{r+1}$ given on the x-axis. The black line is the 97th percentile. The solid blue line with markers is the corresponding error bound. The dotted blue line is a deterministic bound given by (\ref{eq:det1}--\ref{eq:det3}).  The example singular value distribution corresponds to $\sigma_r/\sigma_{r+1}=10^{10}$.}
\label{fig:spread1varygap}
\end{figure}

\begin{figure}
\centering
\includegraphics[scale=\figscale]{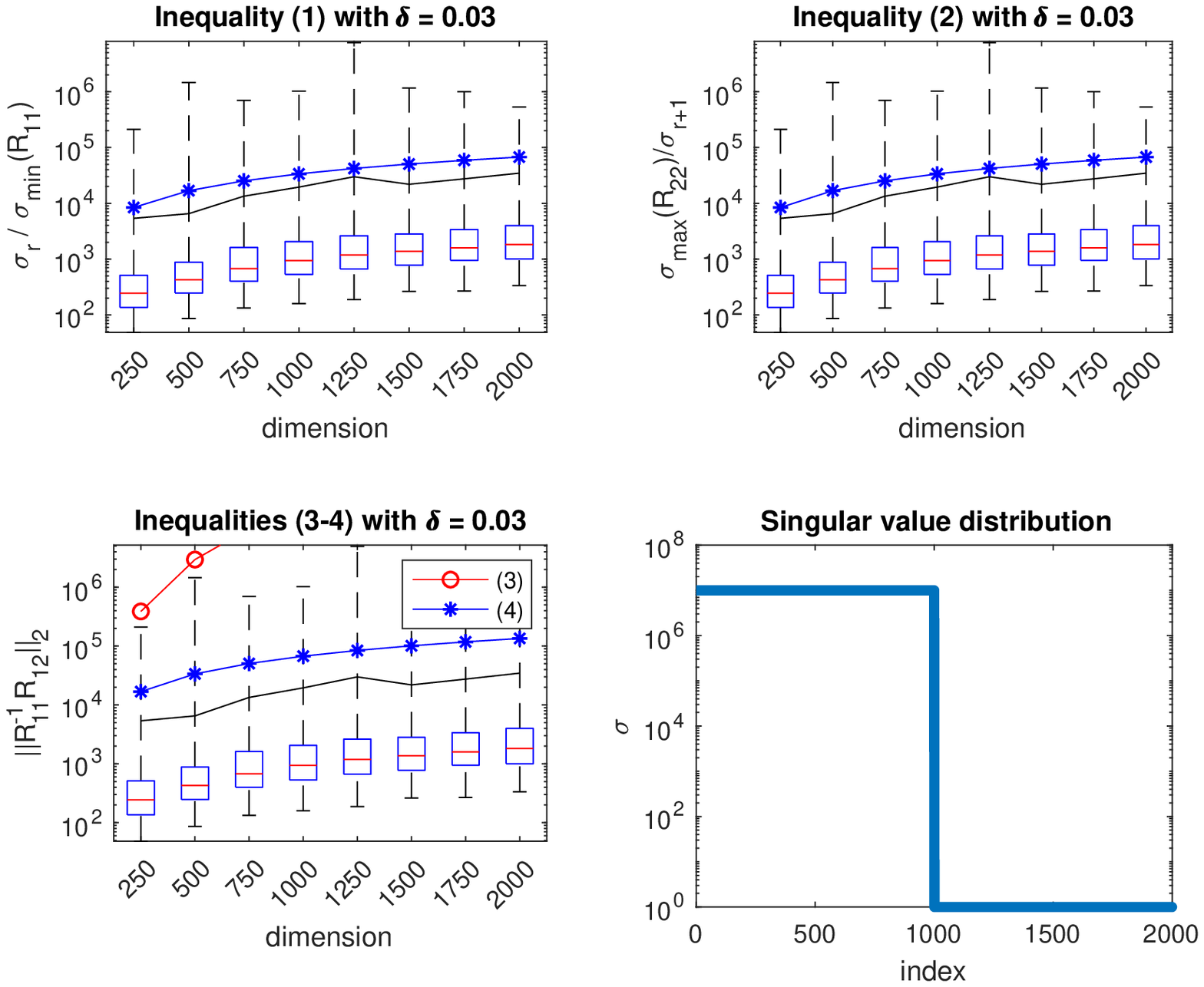}
\caption{Experimental results for test matrices with stair step distribution with  gap $\sigma_r/\sigma_{r+1}=10^7$ and varying dimension given on the x-axis. The black line is the 97th percentile. The solid blue line with markers is the corresponding error bound. The example singular value distribution corresponds to $n=2000$.}
\label{fig:spread1varydim}
\end{figure}

\begin{figure}
\centering
\includegraphics[scale=\figscale]{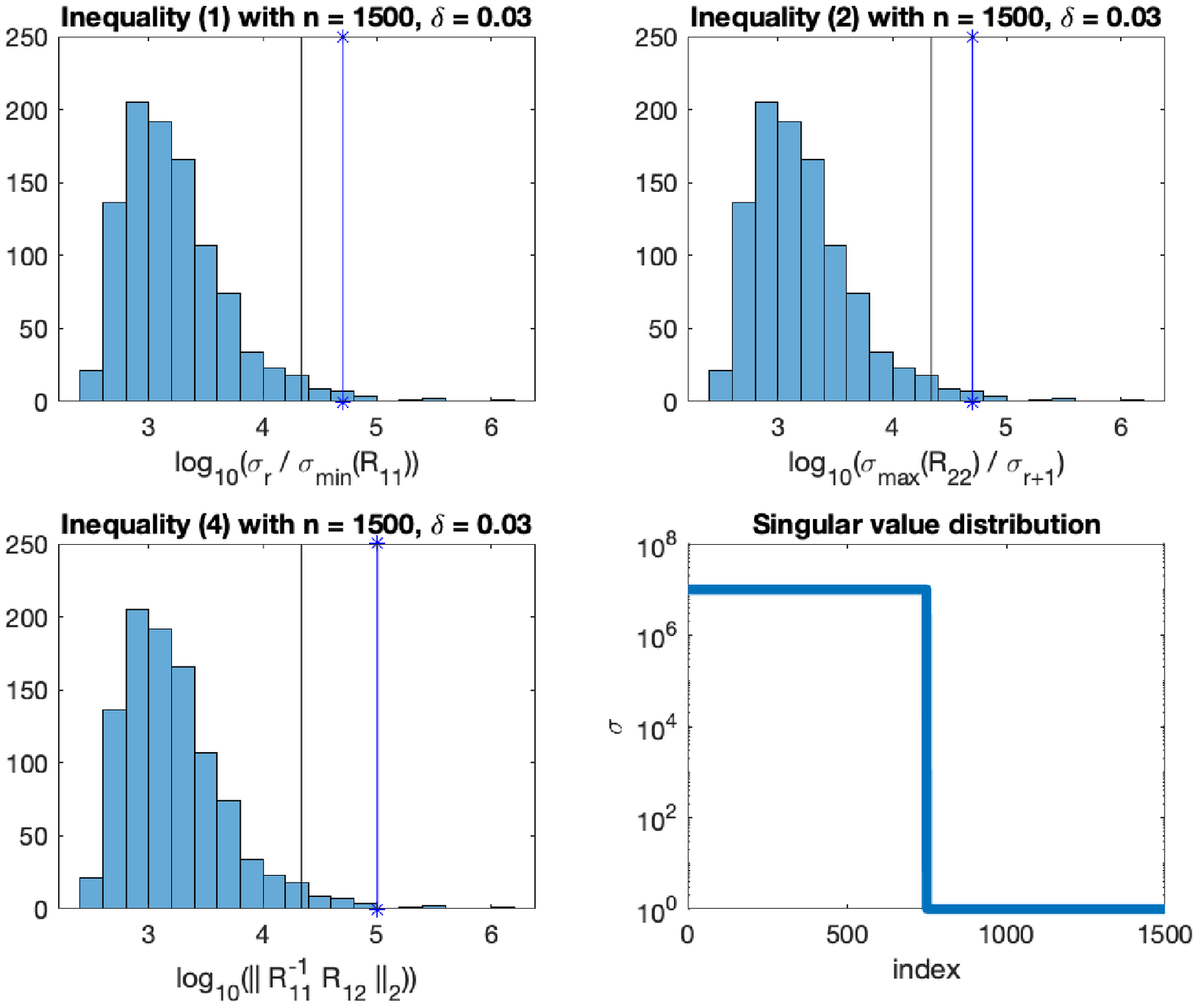}
	\caption{Histograms of 1000 trials of the stair step distribution from $n=1500$ and $\sigma_r/\sigma_{r+1}=10^7$. The black line indicates the 97th percentile. The blue line indicates where our theoretical bounds from \eqref{unu}, \eqref{doi}, and \eqref{four} predict the $97\%$ confidence interval to be.}
\label{fig:spread1varydim-hist}
\end{figure}

\begin{figure}
\centering
\includegraphics[scale=\figscale]{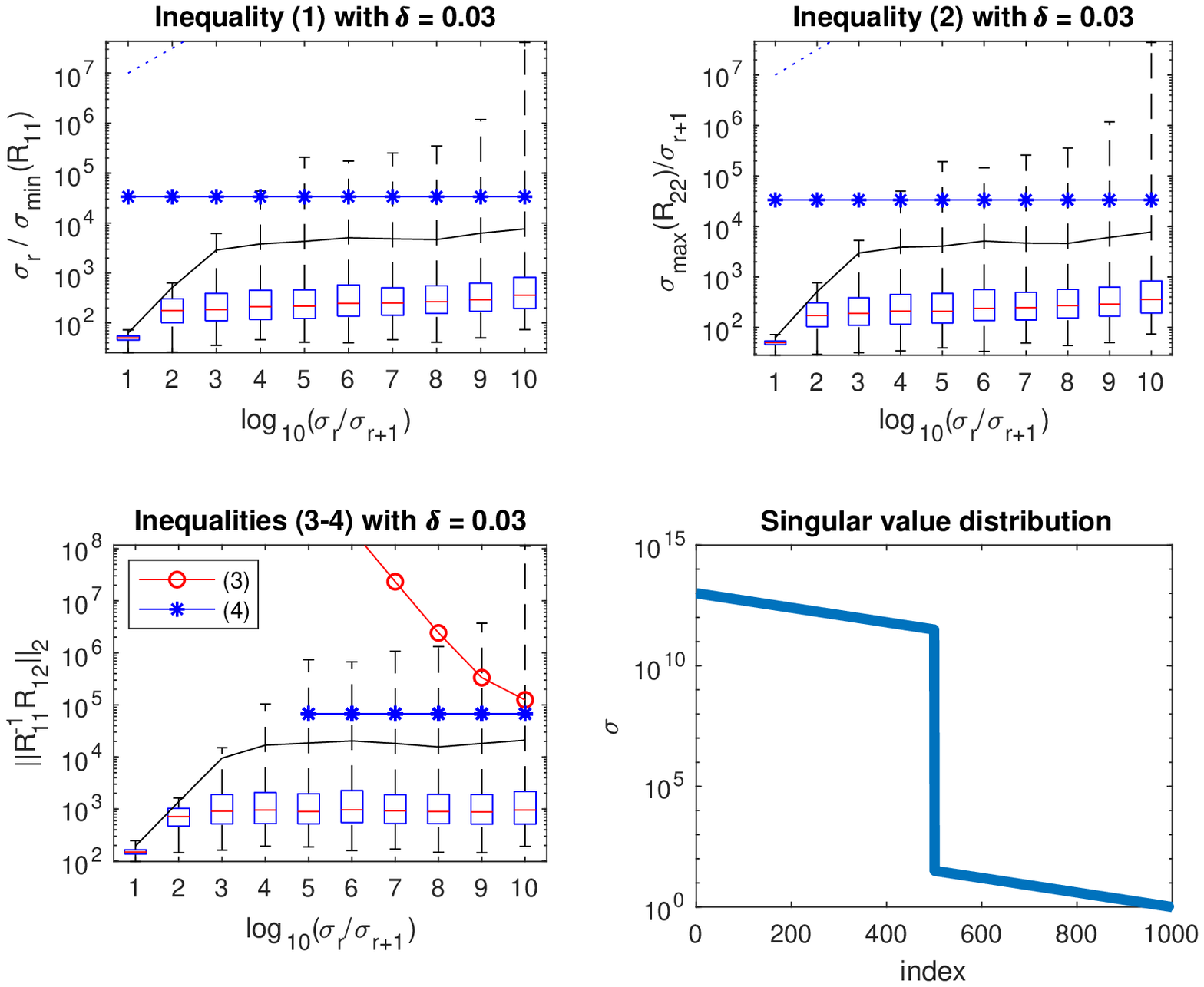}
\caption{Experimental results for test matrices with logarithmically spaced distribution with dimension $n=1500$ and varying gap $\sigma_r/\sigma_{r+1}$ given on the x-axis. The black line is the 97th percentile. The solid blue line with markers is the corresponding error bound. The dotted blue line is a deterministic bound given by (\ref{eq:det1}--\ref{eq:det3}).  The example singular value distribution corresponds to $\sigma_r/\sigma_{r+1}=10^{10}$.   }
\label{fig:spread100varygap}
\end{figure}

\begin{figure}
\centering
\includegraphics[scale=0.7]{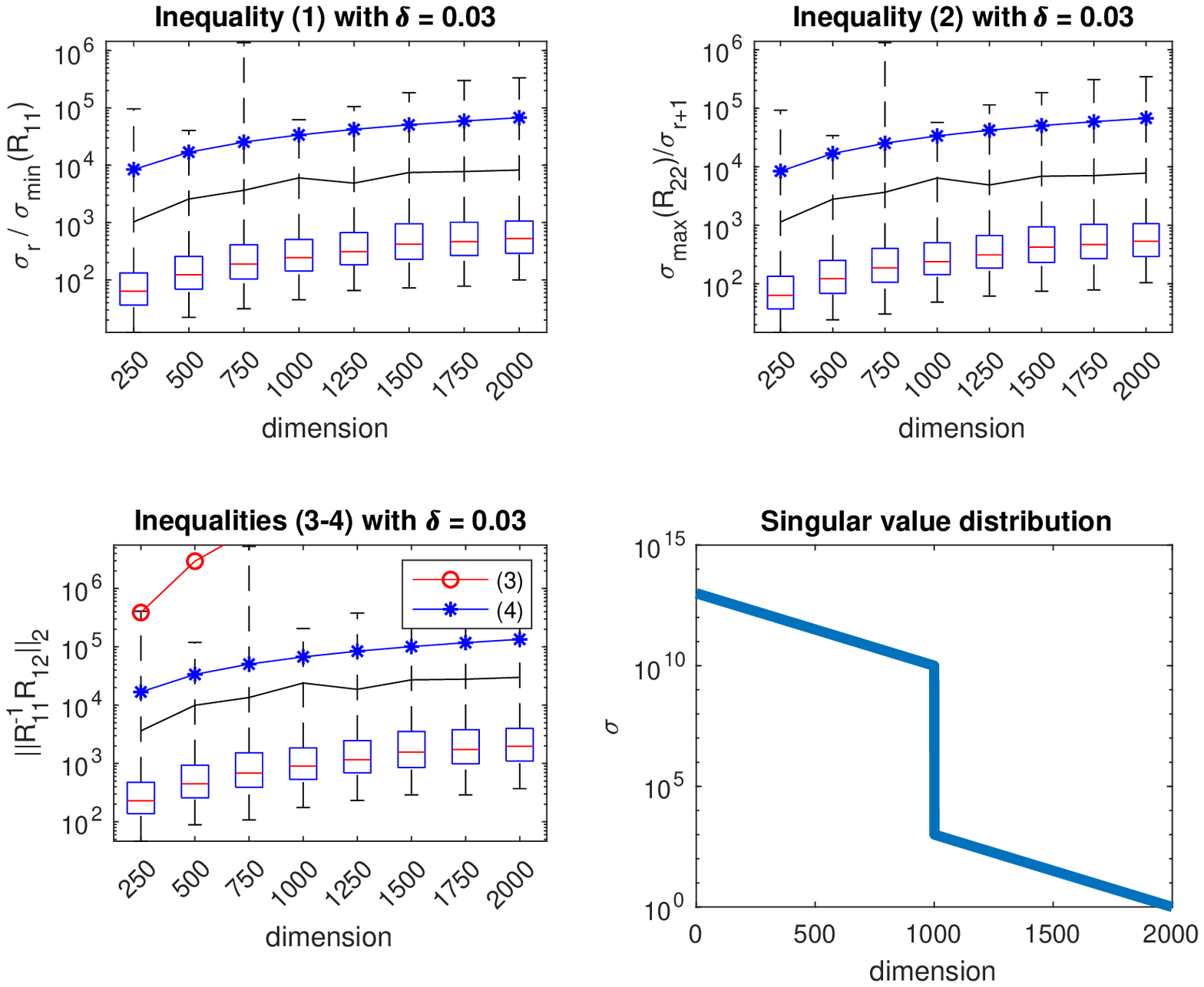}
\caption{Experimental results for test matrices with logarithmically spaced distribution with  gap $\sigma_r/\sigma_{r+1}=10^7$ and varying dimension given on the x-axis. The black line is the 97th percentile. The solid blue line with markers is the corresponding error bound. The example singular value distribution corresponds to $n=2000$.}
\label{fig:spread100varydim}
\end{figure}

\begin{figure}
\centering
\includegraphics[scale=0.7]{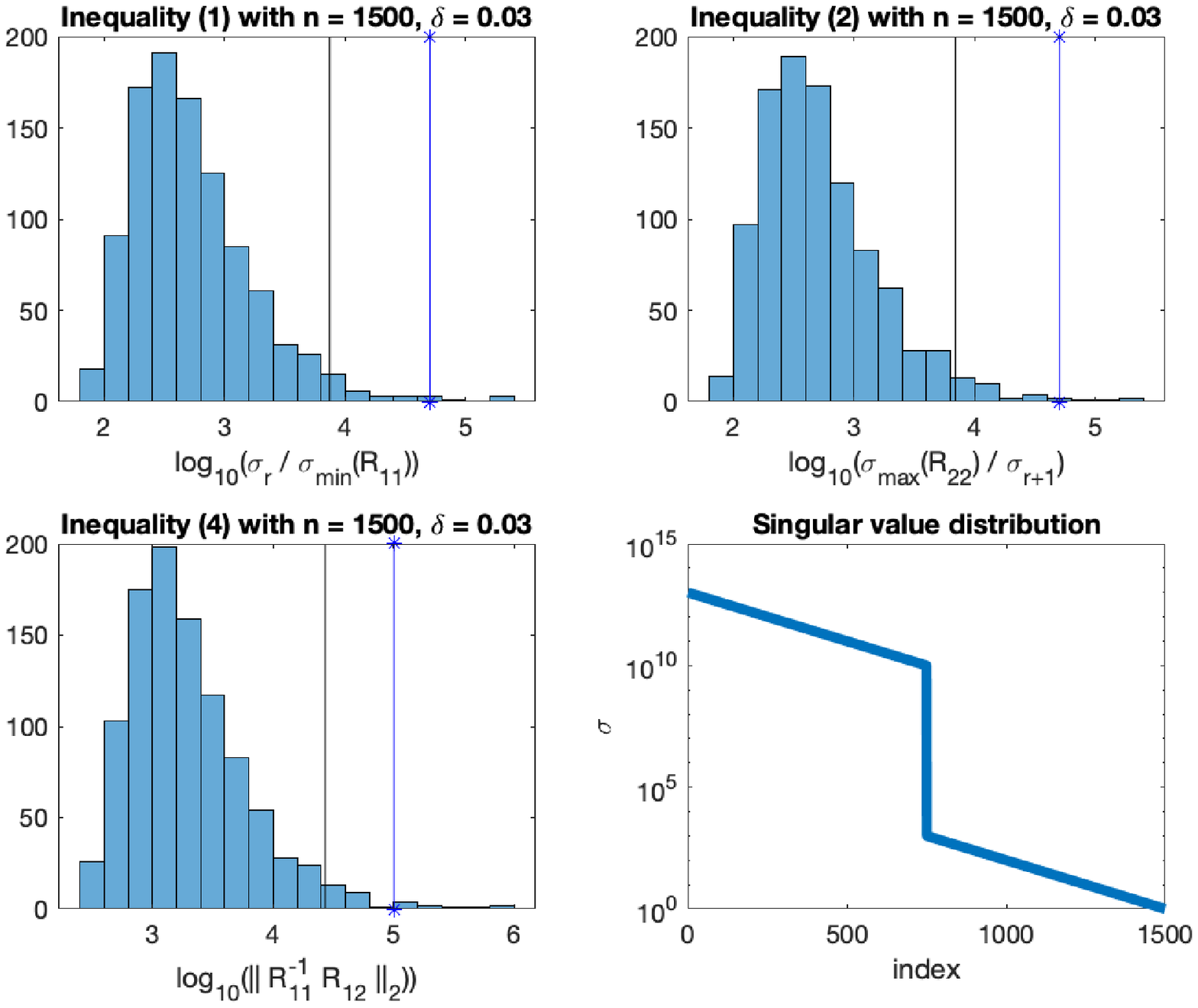}
\caption{Histograms of 1000 trials of the logarithmically spaced distribution from $n=1500$ and $\sigma_r/\sigma_{r+1}=10^7$. The black line indicates the 97th percentile. The blue line indicates where our theoretical bounds from \eqref{unu}, \eqref{doi}, and \eqref{four} predict the $97\%$ confidence interval to be.}
\label{fig:spread100varydim-hist}
\end{figure}

\section{Conclusion}
We have introduced an algorithm for finding the QR-factorization of products of matrices and their inverses, without explicitly computing the products or inverses.  This algorithm is notable for its simplicity, strong theoretical underpinnings, and usefulness as a subroutine within the divide-and-conquer approach to the generalized eigenvalue problem.  Among other important properties, {\bf GURV} was shown to be strongly rank-revealing, backward stable, and communication-optimal.  Moreover, extensive numerical experiments demonstrate that the bounds we presented are essentially tight.

\bibliographystyle{alpha}
\bibliography{rurv}
     
\end{document}